\documentclass[11pt]{amsart}
\usepackage{amsmath,amssymb,latexsym,soul,cite}
\usepackage{color,enumitem,graphicx}
\usepackage[colorlinks=true,urlcolor=blue,
citecolor=red,linkcolor=blue,linktocpage,pdfpagelabels,
bookmarksnumbered,bookmarksopen]{hyperref}
\usepackage[english]{babel}
\usepackage[T1]{fontenc}

\usepackage[left=3cm,right=3cm,top=2.9cm,bottom=2.9cm]{geometry}
\usepackage[hyperpageref]{backref}

\numberwithin{equation}{section}

\newtheorem{thm}{Theorem}[section]
  \theoremstyle{plain}
  \newtheorem{lem}[thm]{Lemma}
  \theoremstyle{plain}
  
  \theoremstyle{plain}
  
  \theoremstyle{remark}
  
  \theoremstyle{definition}

\newcommand{\R}{\mathbb{R}^{N}}

\newcommand{\loc}{{\rm loc}}

\newcommand{\Int}{\displaystyle \int}
\newcommand{\dint}{\displaystyle\int_{\mathbb{R}^{2N}}}
\newcommand{\Frac}{\displaystyle \frac}

\newcommand{\Lim}{\displaystyle \lim}

\renewcommand{\le}{\leqslant}
\renewcommand{\ge}{\geqslant}
\renewcommand{\leq}{\leqslant}
\renewcommand{\geq}{\geqslant}

\title[Fractional $p$-Laplacian problems with weight]{On fractional $p$-Laplacian problems with weight}

\author[R.\ Lehrer]{Raquel Lehrer}
\address{Centro de Ciencias Exatas e Tecnologicas 
\newline\indent 
CCET, Unioeste
\newline\indent
Cascavel-PR, Brazil}
\email{rlehrer@gmail.com}

\author[L.A.\ Maia]{Liliane A. Maia}
\address{Departamento de Matematica
\newline\indent 
Universidade de Brasilia
\newline\indent
Brasilia, Brazil}
\email{lilimaia@unb.br}

\author[M.\ Squassina]{Marco Squassina}
\address{Dipartimento di Informatica
\newline\indent
Universit\`a degli Studi di Verona
\newline\indent
Verona, Italy}
\email{marco.squassina@univr.it}

\subjclass[2000]{34A08, 35Q40, 58E05}
\keywords{$p$-Fractional laplacian, loss of compactness, problems with weight.}
\thanks{The first author is supported by CCET/UNIOESTE. 
The second author is supported by CNPqPQ 306388/2011-1, PROEX/CAPES and FAPDF.
The third author is supported by MIUR project:
   ``{\em Variational and Topological Methods in the Study of Nonlinear Phenomena}''. 
   The work was partially carried out during a stay of Marco Squassina in Brasilia. 
   He would like to express his gratitude to the Departamento de Matem\'atica
   for the warm hospitality.}

\begin{document}

\begin{abstract}
We investigate the existence of nonnegative solutions for a nonlinear problem
involving the fractional $p$-Laplacian operator. The problem is set on a unbounded
domain, and compactness issues have to be handled.
\end{abstract}

\maketitle




\section{Introduction}
The interest for the fractional Laplacian operator $(-\Delta)^s$ and more generally pseudodifferential operators, 
has constantly increased over the last few years, although such operators have been a classical topic of functional analysis since long ago.
Nonlocal operators such as $(-\Delta)^s$ naturally arise
in continuum mechanics, phase transition phenomena,
population dynamics and game theory, as they are the typical outcome of stochastical stabilization of L\'evy processes, see \cite{L,C,MK1} and the references therein. 
We refer the reader to \cite{DiNezza} and to the reference included for a selfcontained overview of the basic properties of fractional Sobolev spaces.
If $\Omega$ is a smooth bounded domain, for semi-linear problems like
\begin{equation*}
\begin{cases}
(-\Delta)^s u = f(x,u), &  \text{in $\Omega$,}  \\
u=0  &  \text{in $\R\setminus\Omega$,} 
\end{cases}
\end{equation*}
existence, nonexistence, regularity and maximum principles have been intensively investigated, see \cite{SV, SV3,CS,RS,RS1,RS2,CS1, CS2}
and the references therein. When $\Omega=\R$, we refer the reader to \cite{FQT,CW} where weak solutions in $H^s(\R)$ are studied.
More recently, for $p>1$, $s\in (0,1)$ and $N>sp$, motivated by some situations arising in game theory,
a nonlinear generalization of this operator has been introduced, see \cite{C,BC}. Precisely, for smooth functions $u$ define
\begin{equation*}
(- \Delta)_p^s\, u(x) := 2\, \lim_{\varepsilon \searrow 0} \int_{\R \setminus B_\varepsilon(x)} \frac{|u(x) - u(y)|^{p-2}\, (u(x) - u(y))}{|x - y|^{N+sp}}\, dy, \quad x \in \R.
\end{equation*}
This nonlinear operator is consistent, up to some normalization constant depending upon $n$ and $s$, 
with the linear fractional Laplacian $(-\Delta)^s$ in the case $p=2$. A broad range of existence and multiplicity results
for the problem
\begin{equation*}
\begin{cases}
(-\Delta)^s_p u = f(x,u), &  \text{in $\Omega$,}  \\
u=0  &  \text{in $\R\setminus\Omega$,} 
\end{cases}
\end{equation*}
has been recently obtained in \cite{ILPS} via tools of Morse theory under different growth assumptions for $f(x,u)$. 
We refer to \cite{FP,LL,IS} for the case $f(x,u)=\lambda |u|^{p-2}u$ and the study of properties of (variational) nonlinear eigenvalues,
including their asymptotic behaviour.
\vskip1pt
\noindent
In this paper, we are concerned with existence of solutions of
\begin{equation}
\begin{cases}
(-\Delta)^{s}_{p} u = \varphi(x)f(u), \quad  \text{in $\R$,} \label{problem} \\
\noalign{\vskip3pt}
u\geq 0, \,\,\, u\neq 0, &
\end{cases}
\end{equation}
In the local case, formally $s=1$, necessary and sufficient conditions for the solvability of 
the problem $-\Delta u=\varphi(x) u^q$ in $\R$  with $0<q<1$ were investigated in \cite{kamin}, see also \cite{brenir}.
Under some sign condition on $\varphi$ the problem with $s=1$ and $p>1$, which thus involves
the $p$-Laplace operator $\Delta_p={\rm div}(|\nabla u|^{p-2}\nabla u)$
was investigated in \cite{ACM}, see also \cite{olimpio}. 
If $F(u):= \int_{0}^{u}f(s)ds,$ the (formal) Poho\v{z}aev identity for solutions $u\in W^{s,p}(\R)$ 
of problem \eqref{problem} is
\begin{equation}
\label{poho}
\int_{\R}\big((N-sp)\varphi(x) f(u)u-pN\varphi(x)F(u)-px\cdot\nabla\varphi(x)F(u)\big)=0.
\end{equation}
A rigorous justification of \eqref{poho} for $p\neq 2$ is still unavailable
due to the lack of suitable regularity results, while in the case $p=2$, \eqref{poho} has been recently proved in 
\cite{RS1}, see also \cite{RS2,CW}. For the case $f(u)=u^q$, the identity yields nonexistence of solutions $u\in W^{s,p}(\R)$ provided that 
$$
\text{$x\mapsto (N-sp)\varphi(x)-\frac{pN}{q+1}\varphi(x)-\frac{p}{q+1}x\cdot\nabla\varphi(x)$ \, has fixed sign in $\R$}.
$$
Then, in particular case where $\varphi$ is constant 
$u=0$ as soon as  $q \neq  p^{*}_s - 1$, where we set
$$
p^*_s:=\frac{Np}{N-sp}.
$$
Hence, in general, it is natural to impose conditions
on $\varphi$ in order to obtain nontrivial solutions.
\vskip3pt
\noindent
We will assume that $p>1$, $\varphi \in L^{\infty}_{\loc}(\R)$ and $f\in C({\mathbb R}^+)$ satisfies the following conditions:
\begin{itemize}
\item[$(f_1)$] $f(s)\geq 0$,\quad\text{for all $s\geq 0$};
\item[$(f_2)$] $\mu s^{q} \leq f(s)\leq c s^{q}$ ,\quad\text{for all $s\geq 0$, some $p-1 < q < p^{*}_{s} - 1$ and $c,\mu>0$};
\item[$(f_3)$] there exists $m<p$ such that 
\begin{align*}
& 0 \leq (q+1)F(s) - f(s)s \leq C s^{m},  \quad\text{for all $s\geq 0$ and some $C>0$}; \\
& 0 \leq f(s)s - pF(s)\leq C s^{q+1} ,\quad\text{for all $s\geq 0$}.
\end{align*}
\item[(W)] $\displaystyle\sup_{\R\setminus\Omega}\varphi \leq 0 < \displaystyle\inf_{\omega}\varphi$ 
for some bounded domains $\omega,\Omega \subset \R$ with $\omega \subset \Omega$.
\end{itemize}
In addition to $f(s):=s^q$ for $s\geq 0$, another example of nonlinearity satisfying $(f_1)$-$(f_3)$ is
$$ 
f(s):= 
\begin{cases}
2s^{q}, & 0 \leq s \leq 1, \\
s^{q} + s^{m-1}, & s \geq 1,
\end{cases}
\qquad m<p<q+1.
$$

\noindent
The main result of the paper is the following:
\begin{thm}
\label{main}
 Assume that $(W)$ and $(f_1)$-$(f_3)$ hold. Then problem $\eqref{problem}$ has a
 distributional solution, namely there exists a function $u\in L^{Np/(N-sp)}(\R)\setminus\{0\}$ with $u\geq 0$,
$$
\dint\Frac{|u(x)-u(y)|^{p}}{|x-y|^{N+sp}}<\infty,
$$
and
\begin{equation*}
\dint\Frac{|u(x) - u(y)|^{p-2}(u(x)-u(y))(\psi(x)-\psi(y))}{|x-y|^{N+sp}} 
= \Int_{\R}\varphi(x)f(u)\psi, 
\end{equation*}
for every $\psi \in C^\infty_c(\R)$. The same holds 
if $(W)$ and $(f_1)$ hold and $(f_2)$ holds with $0 \leq q < p-1$.
\end{thm}

\noindent
We point out that the result is new also for the semi-linear case $p=2$, $1<q<2^*_s-1$ and $N\geq 2$, establishing existence 
of a nonnegative distributional solution $u\in D^{s,2}(\R)$ for 
\begin{equation*}
(-\Delta)^{s} u = \varphi(x)f(u) \quad  \text{in $\R$.} 
\end{equation*}
In general, it is not guaranteed that the {\em distributional} solution $u$ of Theorem~\ref{main} 
belongs to the fractional space $W^{s,p}(\R)$, that is $u\not\in L^p(\R)$ might occur.
Moreover, if the solution $u\geq 0$ of Theorem~\ref{main} is a {\em weak} solution to \eqref{problem}, as pointed
out in \cite[Proposition 2.2]{ILPS} (see also \cite[Theorem A.1]{Brasco} holding for supersolutions
for further details), actually $u>0$ on $\R$.
The proof of Theorem~\ref{main} follows the pattern of \cite{ACM}, namely nontrivial nonnegative
solutions $u_n$ are constructed for the problem defined on a sequence of balls $B(0,R_n)\subset\R$ with $u_n=0$
on $\R\setminus B(0,R_n)$, with $R_n\nearrow\infty$ as $n\to\infty$. Then, relying on uniform 
estimates, the sequence is shown to converge weakly to a nontrivial distributional solution to \eqref{problem}.
Both in getting uniform estimates and in proving the nontriviality of the weak limit, the fact that $\varphi(x)\leq 0$
outside a bounded domain of $\R$ plays a crucial role. 

\section{Preliminary results}
\noindent
The space $D^{s,p}(\R)$ is defined by 
$$
D^{s,p}(\R) := \left\lbrace u \in L^{\frac{Np}{N-sp}}(\R): \, \|u\|_{D^{s,p}} < \infty \right\rbrace, \,\quad
\|u\|_{D^{s,p}}:= \Big(\dint\Frac{|u(x)-u(y)|^{p}}{|x-y|^{N+sp}}\Big)^{1/p}.
$$
Endowed with the norm $\|\cdot\|_{D^{s,p}}$ the space $D^{s,p}(\R)$ is a uniformly convex Banach space. From \cite[Theorem 6.5]{DiNezza}, 
we know that there exists a positive constant $C$ such that
\begin{equation}
\label{fractsob}
\|u \|_{L^{p^*_s}(\R)} \leq C \|u \|_{D^{s,p}},\,\,\quad\text{for every $u\in  D^{s,p}(\R)$},
\end{equation}
and $D^{s,p}(\R)$ is embedded into $L^{q}_{{\rm loc}}(\R)$, for every $1\leq q \leq p^*_s$. 
We observe that, in general, the integral $\varphi F(u)$ may not belong to $L^1(\R)$ for $\varphi \in L^{\infty}_{\loc}(\R)$.
Hence, we shall consider a sequence of diverging radii $R_n>0$ and the spaces 
$$
X_n :=\big\{u \in W^{s,p}(\R): u=0 \ \text{on} \ \R\setminus B(0,R_n)\big\}
$$ 
endowed with the norm 
\begin{equation}
\label{norma}
\|u\|_{X_n}:=\|u\|_{D^{s,p}},\quad \,\, u\in X_n,
\end{equation}
and the functionals $J_n: X_n \rightarrow \mathbb{R}$ given by
$$
J_n(u):=\Frac{1}{p}\dint\Frac{|u(x)-u(y)|^{p}}{|x-y|^{N+sp}} -\Int_{B(0,R_n)}\varphi(x)F(u^{+}), \quad u\in X_n.
$$
We stress that, by means of \eqref{fractsob} and H\"older inequality,
the norm defined in \eqref{norma} is equivalent (with constants depending on the value of $n$)
to the standard norm in $W^{s,p}(\R)$, namely $\|u\|_{W^{s,p}}=(\|u\|_p^p+\|u\|_{D^{s,p}}^p)^{1/p}$.
We can check that $J_n \in C^{1}(X_n, \mathbb{R})$ and, for $u,v \in X_n$, 
$$
J'_{n}(u)(v) = \dint \Frac{|u(x)-u(y)|^{p-2}(u(x)-u(y))(v(x)-v(y))}{|x-y|^{N+sp}} - \Int_{B(0,R_n)}\varphi(x)f(u^{+})v.
$$
The truncation with $u^+:=\max\{u,0\}$ in the nonlinearity will allow critical points of $J_n$ be automatically nonnegative,
see Lemma~\ref{segno}.

\vskip4pt
\noindent
Without loss of generality, we may assume that all the balls $B(0,R_n)$ contain the domain $\Omega$ 
for each $n\geq 1$ large enough.

\begin{lem}
\label{lowerw}
For every $n\geq 1$ the functional $J_n$ is weakly lower semi-continuous on $X_n$.
\end{lem}
\begin{proof}
If $(u_j)\subset X_n$ converges weakly to some $u$ in $X_n$ as $j\to\infty$, we have
$$
\dint\Frac{|u(x)-u(y)|^{p}}{|x-y|^{N+sp}}\leq \liminf_{j\to\infty}  \dint\Frac{|u_j(x)-u_j(y)|^{p}}{|x-y|^{N+sp}}.
$$
Since $(u_j)$ is bounded in $L^p(B(0,R_n))$ via inequality \eqref{fractsob},
the compact embedding theorem for fractional Sobolev spaces \cite[Corollary 7.2]{DiNezza}
implies that, up to a subsequence, the sequence $(u_j)$ converges strongly to $u$ in $L^r(B(0,R_n))$, for every $1\leq r<p^*_s$
and $u_j(x)\to u(x)$ for a.e.\ $x\in\R$. In turn, since by 
condition $(f_2)$ there exists a positive constant $C_n>0$ with
$$
|\varphi(x)F(u_j^+)|\chi_{B(0,R_n)}\leq C_n |u_j|^{q+1},\quad (q+1<p^*_s),
$$
we get by the Dominated Convergence theorem that 
$$
\lim_{j\to\infty}\Int_{B(0,R_n)}\varphi(x)F(u^{+}_j)=\Int_{B(0,R_n)}\varphi(x)F(u^{+}).
$$
This concludes the proof.
\end{proof}

\begin{lem}
\label{segno}
If $J'_{n}(u) = 0$, for $u\in X_n$. Then $u\geq 0$.
\end{lem}
\begin{proof}
Observe first that if $u\in X_n$, then $u^\pm\in X_n$, where $u^\pm:=\max\{\pm u,0\}.$ We have  
\begin{equation}
\label{zerononlin}
\int_{B(0,R_n)}\varphi(x)f(u^{+})u^{-} = 0.
\end{equation}
 We recall the elementary inequality
\[
|\xi^--\eta^-|^p\le|\xi-\eta|^{p-2}(\xi-\eta)(\eta^--\xi^-),\quad\text{for every $\xi,\eta\in {\mathbb R}$}.
\]
Then, recalling \eqref{zerononlin}, by testing $J_n'$ with $-u^-\in X_n$ yields
\begin{align*}
0=J_n'(u)(-u^-)&= \dint \Frac{|u(x)-u(y)|^{p-2}(u(x)-u(y))(u^-(y)-u^-(x))}{|x-y|^{N+sp}} \\
& \geq  \dint\Frac{|u^-(x)-u^-(y)|^{p}}{|x-y|^{N+sp}}.
\end{align*}
This implies that $u^-$ is constant in $\R$ and since $u^-$ vanishes outside $B(0,R_n)$,
it follows that $u^{-}=0$. Hence, $u\ge 0$ a.e., concluding the proof.
\end{proof}

\noindent
In the next two lemmas, we consider the case where $(f_2)$ is satisfied with $q+1 < p$.

\begin{lem}\label{coer}
Assume $(W)$, $(f_1)$ and $(f_2)$ with $q+1 < p$. Then, for each $n\geq 1$, there exists a nonnegative critical point 
$u_n \in X_n\setminus\{0\}$ of $J_n$ such that
$$
J_n(u_n) = \displaystyle\inf_{X_n}J_n < 0. 
$$ 
\end{lem}
\begin{proof}
By virtue of condition $(f_2)$, we have the following inequality
$$
\Int_{B(0,R_n)} \varphi(x)F(u^{+}) \leq c\Int_{B(0,R_n)}\varphi(x)|u^{+}|^{q+1}.
$$
By applying H\"older inequality with $\vartheta := \frac{p^{*}_{s}}{p^{*}_s - (q+1)}$ and $\alpha := \frac{p^{*}_{s}}{q+1}$, we obtain
\begin{align*}
\Int_{B(0,R_n)}\varphi(x)|u^{+}|^{q+1} &\leq   \Vert \varphi\Vert_{L^{\vartheta}(B(0,R_n))} \Vert u \Vert^{q+1}_{L^{p^{*}_{s}}(B(0,R_n))}\\
& =   \Vert \varphi\Vert_{L^{\vartheta}(B(0,R_n))} \Vert u \Vert^{q+1}_{L^{p^{*}_{s}}(\R)}
\leq  C_n \Vert u \Vert^{q+1}_{D^{s,p}},
\end{align*}
for some $C_n>0$. Then, by using this estimate on $J_n$, we obtain
$$
J_n(u)\geq \Frac{1}{p}\Vert u \Vert^{p}_{D^{s,p}} - C_n\Vert u \Vert^{q+1}_{D^{s,p}}.
$$
Since $q+1<p$, and recalling the definition of $\|\cdot\|_{X_n}$,
we conclude that $J_n(u)\to+\infty$ when $\|u\|_{X_n}\rightarrow \infty$, since $p > q+1$, namely $J_n$ is coercive on $X_n$.
Whence, taking into account Lemma~\ref{lowerw}, by a standard argument of the Calculus of Variations, there exists $u_n\in X_n$ such that 
$J_n(u_n) =\inf_{X_n}J_n$, which is a critical point of $J_n$. By Lemma~\ref{segno}, we have $u_n\geq 0$ a.e.\
Now, we take $\zeta \in C^{\infty}_{c}(\R)\setminus\{0\}$ with 
${\rm supp}(\zeta)\subset \omega$. Then using $(f_2)$ again, we obtain
\begin{equation*}
J_n(t\zeta)\leq   \Frac{t^{p}\|\zeta\|^{p}_{D^{s,p}}}{p} - \mu t^{q+1}\Int_{B(0,R_n)}\varphi(x)|\zeta|^{q+1}
=\Frac{t^{p}\Vert \zeta\Vert^{p}_{D^{s,p}}}{p} - \mu t^{q+1}\Int_{\omega}\varphi(x)|\zeta|^{q+1}.
\end{equation*}
Since $\inf_{\omega}\varphi >0$ we have $\int_{\omega}\varphi(x)|\zeta|^{q+1} >0$ and we can conclude that there exists 
$t_n >0$ small enough that $J_{n}(t_n\zeta) < 0.$ Since $t_n\zeta\in X_n,$ we conclude the proof.
\end{proof}

\begin{lem}\label{supremum}
Assume $(W)$, $(f_1)$ and $(f_2)$ with $q+1 <p$. 
Let, for each $n\in \mathbb{N}$, $u_n \in X_n\setminus\{0\}$ 
be the nonnegative critical point of $J_n$ obtained in Lemma~\ref{coer}.
Then there exist two constants $c<0$ and $M>0$, independent of $n$, 
such that:
\begin{itemize}
\item[(i)] $\displaystyle\sup_{n\geq 1}J_n(u_n)\leq c$.
\item[(ii)] $\displaystyle\sup_{n\geq 1} \|u_n\|_{X_n}\leq M$.
\end{itemize}
\end{lem}

\begin{proof}
Taking into account that $u_n\geq 0$, that $\omega \subset \Omega \subset B(0,R_n)$ and by assumption $(W)$, 
$$
\Int_{B(0,R_n)}\varphi(x)F(u_n)=\Int_{\Omega}\varphi(x)F(u_n)+\Int_{B(0,R_n)\setminus\Omega}\varphi(x)F(u_n)
\leq \Int_{\Omega}\varphi(x)F(u_n).
$$
Hence, in turn, we get
\begin{equation}
\label{below-funct}
J_n(u_n) \geq  \Frac{1}{p}\|u_n \|^{p}_{D^{s,p}} - \Int_{\Omega}\varphi(x)F(u_n)
\geq   \Frac{1}{p}\|u_n \|^{p}_{D^{s,p}} - C\|u_n\|^{q+1}_{D^{s,p}},
\end{equation}
where H\"older inequality was used as in the proof of Lemma \ref{coer} but here
the positive constant $C :=\delta\Vert \varphi\Vert_{L^{\vartheta}(\Omega)}$, for some 
$\delta=\delta(\Omega)>0$, is independent of $n\geq 1$.
We also have, by arguing as in the proof of Lemma~\ref{coer}, that
for a $\zeta \in C^{\infty}_{c}(\R)\setminus\{0\}$ with 
${\rm supp}(\zeta)\subset \omega$,
$$
J_n(\tau\zeta) \leq c,\qquad c:=\Frac{\tau^{p}\Vert \zeta\Vert^{p}_{D^{s,p}}}{p} - \mu\tau^{q+1}\Int_{\omega}\varphi(x)|\zeta|^{q+1} < 0 . 
$$
for some $\tau>0$ small enough and independent of $n\geq 1$. 
Thus, we get 
$$
\sup_{n\geq 1}J_n(u_n)=\sup_{n\geq 1}\inf_{X_n}J_n\leq \sup_{n\geq 1}J_n(\tau\zeta)\leq c<0.
$$
This proves (i). By means of inequality \eqref{below-funct},
inequality (ii) immediately follows otherwise a contradiction follows by 
the condition $q+1<p$.
\end{proof}

\noindent
We now turn to the case $p<q+1$, where $\varphi$ and $f$ satisfy $(W)$ and $(f_i)$
respectively.

\begin{lem}\label{MPgeometry}
Assume that $(W)$ and  $(f_1)$-$(f_3)$ hold. 
Then there exist $\rho , r >0$ and a function $\psi\in X_n\setminus\{0\}$, 
independent of $n\geq 1$, with $\|\psi\|_{X_n} >\rho$ such that
\begin{itemize}
\item[(i)] $J_n(u)\geq r$, for every $u\in X_n$ with $\| u \|_{X_n} = \rho$ and all $n\geq 1$;
\item[(ii)] $J_n(\psi)\leq 0$, for all $n\geq 1$.
\end{itemize} 
\end{lem}

\begin{proof}
We have, arguing as in Lemma \ref{supremum}, that for all $u\in X_n$
\begin{equation*}
J_n(u) \geq  \Frac{1}{p}\Vert u \Vert^{p}_{D^{s,p}} - C \Vert u\Vert^{q+1}_{D^{s,p}},
\end{equation*}
with $C$ independent of $n\geq 1$. Take $\rho>0$ such that $\rho^{q-p+1} < 1/2pC$. Then, if $\|u\|_{D^{s,p}} = \rho$, we obtain
$J_n(u)\geq r$, with $r:=\rho^{p}/2p >0.$ On the other hand, as in the proof of Lemma \ref{supremum}, there exists 
some $t_0 >0$ (this time large enough) independent of $n\geq 1$ such that 
$J_n(t_0\zeta) \leq 0$ and taking $\psi:= t_0\zeta$ we have $J_n(\psi)\leq 0$. Up to reducing $\rho$, 
we also get $\| \psi\|_{D^{s,p}} = t_0\| \zeta \|_{D^{s,p}} > \rho$.
\end{proof}

\noindent
By Lemma~\ref{MPgeometry}, we can define, for each $n\geq 1$,
the min-max level for $J_n$:
$$c_n:= \displaystyle\inf_{\gamma \in \Gamma_n}\max_{0\leq t \leq 1}J_n(\gamma(t)),\qquad
\Gamma_n := \left\lbrace \gamma \in C([0,1], X_n); \gamma(0)=0, \gamma(1)=\psi\right\rbrace.
$$
Using the fact that $X_n \subset X_{n+1}$ we actually have 
$$
c_1 \geq c_2 \geq \cdots \geq c_n \geq \cdots \geq r >0,
$$
so that in particular $c_n \to c$, for some $c \geq r > 0.$

\begin{lem} 
\label{PSC}
Assume that $(W)$ and  $(f_1)$-$(f_3)$ hold. Then the functional $J_n$ satisfies 
the $(PS)_c$-condition, for every $c \in \mathbb{R}$ and for all $n\geq 1$.
\end{lem}

\begin{proof}
Suppose now that $J_n(u_j)\to c$ and $J'_n(u_j)\to 0$ as $j\to\infty$. 
Then we can write
\begin{align}
c + o_j(1) &= \Frac{\Vert u_j\Vert^{p}_{D^{s,p}}}{p} - \Int_{B(0,R_n)}\varphi(x)F(u^{+}_j)\label{Jn},  \\
 o_j(1)\Vert u_j\Vert_{D^{s,p}} &= \Vert u_j\Vert^{p}_{D^{s,p}} - \Int_{B(0,R_n)}\varphi(x)f(u^{+}_j)u_j. \label{Jnprime}
 \end{align}
By combining these identities, we obtain
 $$
 \Big(\Frac{1}{p} - \Frac{1}{q+1} \Big)\Vert u_j\Vert^{p}_{D^{s,p}} - \Int_{B(0,R_n)}\varphi(x)\Big( F(u^{+}_j) - \Frac{f(u^{+}_j)u_j}{q+1}\Big) = c + o_j(1) +o_j(1)\|u_j\|_{D^{s,p}}.
 $$
 In turn, on account of condition $(f_3)$, we have
 \begin{align*}
 \Int_{B(0,R_n)}\varphi(x)\Big( F(u^{+}_j) - \Frac{f(u^{+}_j)u_j}{q+1}\Big) &\leq  \|\varphi\|_{L^\infty(B(0,R_n))}\Int_{B(0,R_n)}\Big|F(u^{+}_j) - \Frac{f(u^{+}_j)u_j}{q+1}\Big|\ \\
 &\leq  C_n\Int_{B(0,R_n)}|u_j|^{m} \leq C_{n}\Vert u_j \Vert^{m}_{L^{p^{*}_{s}}} \leq C_{n}\Vert u_j\Vert^{m}_{D^{s,p}}.
 \end{align*}
 Therefore, we get
 \begin{equation*}
 c + o_j(1) +o_j(1)\Vert u_j\Vert_{D^{s,p}}  \geq  \Big(\Frac{1}{p} - \Frac{1}{q+1}\Big)\Vert u_j\Vert^{p}_{D^{s,p}}- C_{n}\Vert u_j\Vert^{m}_{D^{s,p}}.
 \end{equation*}
 Since $p>m$ and $q+1>p$, this implies that there exists $C(s,n,p,q,c)>0$ such that 
 $$ 
\sup_{j\geq 1} \|u_j\|_{D^{s,p}} \leq C(s,n,p,q,c),
 $$
 namely  the sequence $(u_j)$ is bounded in $D^{s,p}(\R)$. In turn, there exists a subsequence, still denoted by 
 $(u_j)$, such that $u_j \rightharpoonup u$ in $X_n$ as $j\to\infty$. We also have that $u_j \rightarrow u$ in $L^{r}(B(0,R_n))$, 
 for any $1\leq r < p^{*}_{s}$ by the compact embedding theorem \cite[Corollary 7.2]{DiNezza}
 and $u_j(x)\to  u(x)$ for a.e. $x\in\R$. For any $\psi \in X_n$, we have
\begin{equation*}
\dint\Frac{|u_j(x) - u_j(y)|^{p-2}(u_j(x)-u_j(y))(\psi(x) - \psi(y))}{|x-y|^{N+sp}} = \Int_{B(0,R_n)}\varphi(x)f(u^{+}_j)\psi + 
\langle J_n'(u_j), \psi\rangle.
\end{equation*}
For each $\psi \in X_n$ fixed, we have by dominated convergence
$$
\Lim_{j\to \infty}\Int_{B(0,R_n)}\varphi(x)f(u^{+}_j)\psi = \Int_{B(0,R_n)}\varphi(x)f(u^{+})\psi,
$$
since there exists $\eta\in L^{q+1}(\R)$ such that $|u_j|\leq \eta$ a.e.\ and, for some $C_n>0$,
$$
|\varphi(x)f(u^{+}_j)\psi \chi_{B(0,R_n)}|\leq C_n |u^{+}_j|^{q}|\psi|\leq C_n |\eta|^{q}|\psi| \in L^1(\R),\quad\,\,\text{for all $j\geq 1$}.
$$
Now, if $p'$ is the conjugate exponent to $p$, we have 
$$
\text{the sequence\, $\left(\Frac{|u_j(x) - u_j(y)|^{p-2}(u_j(x)-u_j(y))}{|x-y|^{(N+sp)/p'}}\right)$
\, is bounded in $L^{p'}(\mathbb{R}^{2n})$}
$$
as well as 
$$
\Frac{|u_j(x) - u_j(y)|^{p-2}(u_j(x)-u_j(y))}{|x-y|^{(N+sp)/p'}} \rightarrow \Frac{|u(x) - u(y)|^{p-2}(u(x)-u(y))}{|x-y|^{(N+sp)/p'}} \ \,\,\,\,\text{a.e. in $\mathbb{R}^{2n}$}.
$$
Also, since $(\psi(x)-\psi(y))/|x-y|^{(N+sp)/p}\in L^p(\mathbb{R}^{2n})$
we have (cf.\ \cite[Lemma 4.8]{kavian}) that
$$
\dint\Frac{|u_j(x) - u_j(y)|^{p-2}(u_j(x)-u_j(y))(\psi(x) - \psi(y))}{|x-y|^{N+sp}}
$$
converges to
$$
\dint\Frac{|u(x) - u(y)|^{p-2}(u(x)-u(y))(\psi(x) - \psi(y))}{|x-y|^{N+sp}}.
$$ 
This shows that $u\in X_n$ is a weak solution in $B(0,R_n)$, namely 
\begin{equation}
\label{fissa}
\dint\Frac{|u(x) - u(y)|^{p-2}(u(x)-u(y))(\psi(x)-\psi(y))}{|x-y|^{N+sp}} 
= \Int_{B(0,R_n)}\varphi(x)f(u^{+})\psi, \quad\forall \psi \in X_n.
\end{equation}
Choosing $\psi = u$ in~\eqref{fissa} and $\psi = u_j$ in the above equation for $J_n'(u_j)$ and since for $C_n>0$,
$$
|\varphi(x)f(u^{+}_j)u_j \chi_{B(0,R_n)}|\leq C_n |u_j|^{q+1}\leq C_n |\eta|^{q+1} \in L^1(\R),\quad\,\,\text{for all $j\geq 1$},
$$
we obtain
\begin{align*}
\Vert u\Vert^{p}_{D^{s,p}} &= \Int_{B(0,R_n)}\varphi(x)f(u^{+})u = \Lim_{j\to \infty}\Int_{B(0,R_n)}\varphi(x)f(u^{+}_j)u_j\\
&=\Lim_{j\to \infty}\dint\Frac{|u_j(x) - u_j(y)|^{p}}{|x-y|^{N+sp}} = \Lim_{j\to \infty}\Vert u_j\Vert^{p}_{D^{s,p}}
\end{align*}
Since also $u_j \rightharpoonup u$, we can conclude that $u_j \to u$ in $X_n$, concluding the proof.
 \end{proof}

\noindent
We can finally state the following

\begin{lem}
\label{ex}
Assume that $(W)$ and  $(f_1)$-$(f_3)$ hold. Then, for each $n\geq 1$, the problem 
\begin{equation}
\label{BallP}
\begin{cases}
(-\Delta)^{s}_{p} u = \varphi(x)f(u), &  \text{in $B(0,R_n)$,}  \\
\noalign{\vskip3pt}
u=0,  &  \text{in $\R\setminus B(0,R_n)$,}
\end{cases}
\end{equation}
admits a nontrivial nonnegative solution $u_n\in X_n$.
\end{lem}
\begin{proof}
By Lemmas~\ref{segno}, \ref{MPgeometry} and \ref{PSC}, the assertion follows
by the Mountain Pass Theorem.
\end{proof}

\section{Proof of Theorem~\ref{main}}  
\noindent
Consider first the case where $(W)$ and  $(f_1)$-$(f_3)$ hold. By virtue of Lemma~\ref{ex}, there exists a sequence $(u_n)\subset X_n \subset D^{s,p}(\R)$ 
 of nontrivial nonnegative weak solutions to problem~\eqref{BallP} on the exhausting balls $B(0,R_n)$, namely
 \begin{equation}
 \label{prob-n-sol}
 \dint\Frac{|u_n(x) - u_n(y)|^{p-2}(u_n(x)-u_n(y))(\psi(x) - \psi(y))}{|x-y|^{N+sp}}= \Int_{B(0,R_n)}\varphi(x)f(u_n)\psi,
 \end{equation}
 for any $\psi \in D^{s,p}(\R)$, with $\psi \equiv 0$ on $\R \setminus B(0,R_n)$. We claim that this sequence
 remains bounded in $D^{s,p}(\R)$. In fact, for every $n\geq 1$, we can write
\begin{equation*}
\Frac{\Vert u_n\Vert^{p}_{D^{s,p}}}{p} - \Int_{B(0,R_n)}\varphi(x)F(u_n)=c_n,  \quad
\Vert u_n\Vert^{p}_{D^{s,p}} - \Int_{B(0,R_n)}\varphi(x)f(u_n)u_n=0.
 \end{equation*}
By combining these identities, we obtain
 $$
 \Big(\Frac{1}{p} - \Frac{1}{q+1} \Big)\Vert u_n\Vert^{p}_{D^{s,p}} - \Int_{B(0,R_n)}\varphi(x)\Big( F(u_n) - \Frac{f(u_n)u_n}{q+1}\Big) = c_n.
 $$
 In turn, on account of conditions $(f_3)$ and $(W)$, we have
 \begin{align*}
 &\Int_{B(0,R_n)}\varphi(x)\Big( F(u_n) - \Frac{f(u_n)u_n}{q+1}\Big)=\Int_{\Omega}\varphi(x)\Big( F(u_n) - \Frac{f(u_n)u_n}{q+1}\Big)  \\
& \qquad +\Int_{B(0,R_n)\setminus\Omega}\varphi(x)\Big( F(u_n) - \Frac{f(u_n)u_n}{q+1}\Big)  
\leq \Int_{\Omega}\varphi(x)\Big( F(u_n) - \Frac{f(u_n)u_n}{q+1}\Big) \\
 & \qquad  \leq  \|\varphi\|_{L^\infty(\Omega)}\Int_{\Omega}\Big|F(u_n) - \Frac{f(u_n)u_n}{q+1}\Big|
 \leq  C\Int_{\Omega}|u_n|^{m} \leq C\Vert u_n \Vert^{m}_{L^{p^{*}_{s}}} \leq C\Vert u_n\Vert^{m}_{D^{s,p}}.
 \end{align*}
where $C=C(\Omega)$ is independent of $n\geq 1$, Therefore, we can conclude that
 \begin{equation*}
 c_1\geq c_n\geq  \Big(\Frac{1}{p} - \Frac{1}{q+1}\Big)\Vert u_n\Vert^{p}_{D^{s,p}}- C\Vert u_n\Vert^{m}_{D^{s,p}}.
 \end{equation*}
 Since $p>m$ and $q+1>p$, the claim is proved. Then, there exists a subsequence, still denoted by 
$(u_n)$, such that $u_n \rightharpoonup u$ in $D^{s,p}(\R)$ as $n\to\infty$. We also have $u_n \rightarrow u$ in $L^{r}(K)$ 
 for any bounded subset $K\subset \R$ and all $1\leq r < p^{*}_{s}$ by the compact embedding theorem  
 \cite[Corollary 7.2]{DiNezza} and $u_n(x)\to  u(x)$ for a.e.\ Arguing as in the proof of Lemma~\ref{PSC}, it follows that $u$
 is a distributional weak solution to problem \eqref{problem}.
In fact, let $\psi \in C^\infty_c(\R)$ and set $K:={\rm supt}(\psi)$. Then
$\psi \in D^{s,p}(\R)$ and $\psi \equiv 0$ on $\R \setminus B(0,R_n)$, for $n\geq 1$ large enough.
The left-hand side of \eqref{prob-n-sol} converges as in the proof of Lemma~\ref{PSC}, by means of duality arguments.
As far as the right-hand side is concerned, by dominated convergence, we get
 \begin{align*}
 \Lim_{n\to \infty} \Int_{B(0,R_n)}\varphi(x)f(u_n)\psi  
&=  \Lim_{n\to \infty} \Int_{K}\varphi(x)f(u_n)\psi  \\
&
= \Int_{K}\varphi(x)f(u)\psi= \Int_{\R}\varphi(x)f(u)\psi,
 \end{align*}
 since there exists $\eta\in L^q(K)$ such that $u_n\leq \eta$ a.e.\ in $K$ for all $n\geq 1$ and
 $$
 |\varphi(x)f(u_n)\psi|\chi_K(x)\leq C u_n^{q}\chi_K(x)\leq C \eta^q\chi_K(x) \in L^1(K).
 $$
 We will now show that $u \neq 0$. Taking $(W)$ and $(f_3)$
 into account, we deduce that
\begin{align*}
c_n &=  \Frac{1}{p}\Vert u_n\Vert^{p}_{D^{s,p}} - \Int_{B(0,R_n)}\varphi(x)F(u_n)
 =   \Int_{B(0,R_n)}\varphi(x)\Big(\Frac{f(u_n)u_n}{p} - F(u_n)\Big)\\
& =   \Int_{\Omega}\varphi(x)\Big(\Frac{f(u_n)u_n}{p} - F(u_n)\Big) +  \Int_{B(0,R_n)\setminus \Omega}\varphi(x)\Big(\Frac{f(u_n)u_n}{p}- F(u_n)\Big)\\
&\leq   \Int_{\Omega}\varphi(x)\Big(\Frac{f(u_n)u_n}{p} - F(u_n)\Big).
\end{align*} 
Since $(c_n)$ is monotone and bounded from below by $r>0$,
we have $c_n\to c>0$ as $n\to\infty$. Whence, assuming by contradiction that $u\equiv 0$, since $\Omega$ is bounded and, for all $n\geq 1$,
$$
0\leq (f(u_n)u_n - pF(u_n))\chi_\Omega\leq C u_n^{q+1}\chi_\Omega\leq C \eta^{q+1}\chi_\Omega,
$$
for some $\eta\in L^{q+1}(\Omega)$, we would obtain by dominated convergence
$$
\Lim_{n\to \infty}\Int_{\Omega}\varphi(x)\Big(\Frac{1}{p}f(u_n)u_n - F(u_n)\Big)= \Int_{\Omega}\varphi(x)\Big(\Frac{1}{p}f(u)u - F(u)\Big) = 0,
$$
which is a contradiction. Therefore $u\neq 0$ and the proof is complete under the first assumption.
If instead condition $(f_2)$  holds for $0 \leq q < p-1$, then the proof follows in a  similar way, in light of Lemma~\ref{coer} and 
Lemma~\ref{supremum}. \qed

\bigskip
\bigskip

\end{document}